\documentclass[12pt]{article}

\usepackage[margin=1in]{geometry}
\usepackage{diagbox}
\usepackage{mathtools}
\usepackage{bbm}
\usepackage{latexsym}
\usepackage{epsfig}
\usepackage{amsmath,amsthm,amssymb,enumerate}
\usepackage[ruled,vlined]{algorithm2e}
\usepackage{comment}

\usepackage[a-1b]{pdfx}
\usepackage{hyperref}

\usepackage{color}
\def\red{\color{black}}
\def\blue{\color{black}}
\def\brown{\color{black}}
\def\cI{{\mathcal I}}
\def\cH{{\mathcal H}}
\def\nn{\nonumber}
\def\a{\alpha} \def\b{\beta} \def\d{\delta} \def\D{\Delta}
\def\e{\varepsilon} \def\f{\phi}   \def\g{\gamma}
  
\def\z{\zeta}     \def\l{\lambda}
  \def\n{\nu} 
\def\r{\rho}  \def\s{\sigma} 
\def\t{\tau} \def\om{\omega}

\newtheorem{theorem}{Theorem}
\newtheorem{lemma}[theorem]{Lemma}
\newtheorem{corollary}[theorem]{Corollary}
\newtheorem{example}[theorem]{Example}

\newtheorem{definition}{Definition}

\newcommand{\rdown}[1]{{\left\lfloor #1\right \rfloor}}

\newcommand{\brac}[1]{\left(#1\right)}

\newcommand{\bfrac}[2]{\left(\frac{#1}{#2}\right)}

\def\cE{{\cal E}}

\newcommand{\set}[1]{\left\{#1\right\}}
\def\sm{\setminus}

\def\es{\emptyset}

\def\E{\mathbb{E}}
\def\Var{{\mathbb V}ar}
\def\Pr{\mathbb{P}}

\newcommand{\ignore}[1]{}

\def\cA{{\mathcal A}}

\def\cD{{\mathcal D}}
\def\cE{{\mathcal E}}

\def\cH{{\mathcal H}}
\def\cI{{\mathcal I}}
\def\cJ{{\mathcal J}}

\def\cO{{\mathcal O}}

\def\cS{{\mathcal S}}

\newcommand{\card}[1]{\left|#1\right|}

\newcommand{\beq}[2]{\begin{equation}\label{#1}#2\end{equation}}
\newcommand{\mults}[1]{\begin{multline*}#1\end{multline*}}
\newcommand{\mult}[2]{\begin{multline}\label{#1}#2\end{multline}}

\def\nn{\nonumber}

\def\cI{{\cal I}}
\def\cJ{{\cal J}}

\usepackage{tikz}
\usetikzlibrary{decorations.pathmorphing}
\usetikzlibrary{positioning}
\usetikzlibrary{arrows,automata}
\usetikzlibrary{shapes.misc}
\usetikzlibrary{backgrounds}
\usetikzlibrary{arrows,shapes}

\newcommand{\pbc}[1]{{\color{blue}{\bf [~Patrick:\ } \color{blue}{\em #1}\color{blue}{\bf~]}}}

\begin{document}
\author{Patrick Bennett\thanks{Department of Mathematics, Western Michigan University, Kalamazoo MI49008}\and Alan Frieze\thanks{Department of Mathematical Sciences, Carnegie Mellon University, Pittsburgh PA 15213, Research supported in part by NSF grant DMS1952285}\and Andrew Newman\thanks{Department of Mathematical Sciences, Carnegie Mellon University, Pittsburgh PA 15213.}\and Wesley Pegden\thanks{Department of Mathematical Sciences, Carnegie Mellon University, Pittsburgh PA 15213, Research supported in part by NSF grant DMS1700365. }}

\title{On the intersecting family process}
\maketitle
\begin{abstract}
We study the intersecting family process initially studied in \cite{BCFMR}. Here $k=k(n)$ and  $E_1,E_2,\ldots,E_m$ is a random sequence of $k$-sets from $\binom{[n]}{k}$ where $E_{r+1}$ is uniformly chosen from those $k$-sets that are not already chosen and that meet $E_i,i=1,2,\ldots,r$. We prove some new results for the case where $k=cn^{1/3}$ and for the case where $k\gg n^{1/2}$.
\end{abstract}
\section{Introduction}

We study the following process introduced by Bohman, Cooper, Frieze, Martin and Ruszink\'o \cite{BCFMR}: consider the random sequence $\cI_k=(E_1,E_2,\ldots,E_m)$ where $E_i\in\binom{[n]}{k}$ for $i=1,2,\ldots,m$ and (i) $E_1$ is uniformly random and (ii) $E_{i+1}$ is randomly chosen from the $k$-subsets of $[n]$ that are not already chosen and that intersect each of $E_1,E_2,\ldots,E_i$. The process continues until no further sets can be added i.e. until $\{E_1,E_2,\ldots,E_m\}$ is a maximal intersecting family. We will abuse terminology and sometimes consider $\cI_k$ to be a set of edges (i.e. a $k$-uniform hypergraph) instead of a sequence of edges.

We denote the hypergraph comprising the first $r$ accepted edges by $\cH_r$. For a set $S\subseteq [n]$, we let $e_r(S)$ denote the number of edges $E_i,i\leq r$ such that $E_i\cap S\neq \emptyset$.
For our purposes, an intersecting family is {\em trivial} if it is of the form $\cA_x=\set{E\in \binom{[n]}{k}:x\in E}$ where $x$ is some fixed element of $[n]$. Now $|\cA_x|=\binom{n-1}{k-1}$ and the famous result of Erd\H{o}s, Ko and Rado \cite{EKR} is that if $k\leq n/2$ then any intersecting family of $k$-sets is bounded in size by $\binom{n-1}{k-1}$ and that the maximum is achieved only by trivial families. The aim of \cite{BCFMR} was to see when the process $\cI_k$ produces a trivial family w.h.p. Their main result is the following:
\begin{theorem}\label{main}\
Let $\cE_0$ be the event that $\cI_k=\cA_x$ for some $x\in [n]$. If $k=c_nn^{1/3}<n/2$ then
$$\lim_{n\to\infty}\Pr(\cE_0)=\begin{cases}1&c_n\to 0\\
\frac{1}{1+c^3}&c_n\to c\\0&c_n\to\infty\end{cases}.$$
\end{theorem}
So if $c_n\to c$ in this theorem, then in the limit, there is a positive probability of $\frac{c^3}{1+c^3}$ that $\cI_k$ is not trivial. What can be said about this case? 

Patk\'os \cite{Pat} considered this question and showed that in the random intersecting process we study here that for $k = c_nn^{1/3}$ with $c_n \rightarrow c$ with probability 
\[\left(\frac{c^3}{1 + c^3}\right) \left(\frac{3}{3 + c^3} \right)\]
$\cI_k$ is a \emph{Hilton--Milner}-type hypergraph. A Hilton--Milner-type hypergraph was first described in \cite{HM} and is a $k$-uniform hypergraph obtained by specifying a single vertex $v$ and a single edge $F$ that does not contain $v$ and the edges are $F \cup \{E \in \binom{[n]}{k} \mid v \in E, F \cap E \neq \emptyset\}$. Our goal here is to extend beyond the two possibilities considered so far, the trivial system and the Hilton--Milner system, to further understand the full distribution of the asymptotic behavior in the critical regime.

Our first result (Theorem \ref{thm2} below) applies to the case where $k=c_n n^{1/3}$ where $c_n \rightarrow c$. We describe two randomly generated families $\cI^*, \cI^{**}$, which are approximately the same size, such that w.h.p. $\cI^* \subseteq \cI_k \subseteq \cI^{**}$. Furthermore $\cI^*, \cI^{**}$ can be determined from the very early evolution of the process.

We define two hitting times which will help determine $\cI^*, \cI^{**}$. We let $r_0$ be the first step $r$ such that $\cH_r$ has a vertex of degree three. We let $J$ be the set of vertices of degree two in $\cH_{r_0-1}$. A set $S\subseteq J$ is said to be {\em independent} if no edge of $\cH_{r_0-1}$ contains more than one member of $S$. For $r \ge r_0$ we let $\cS_r$ be the family of all $S \subseteq J$ that are independent  and which meet every edge $E_{r_0}, \ldots, E_{r}$. Let $r_1$ be the first step $r$ when $\cS_r$ is an intersecting family, and set $\cS:=\cS_{r_1}$ to be that intersecting family. (We show below that w.h.p. $\cS_r$ becomes an intertsecting family before the process terminates.) To summarise:{\red
\begin{align*}
r_0&=\text{the first step $r$ such that $\cH_r$ has a vertex of degree three.}\\
r_1&=\text{the first step $r$ when $\cS_r$ is an intersecting family.}
\end{align*}}
{\blue In the limit as $n \rightarrow \infty$, $\cS$ has a natural interpretation as an intersecting family of matchings of a complete graph. Both the size of this complete graph and the actual intersecting family are random, and in general $\cS$ will not be a uniform hypergraph. In the appendix we explain this random process to generate $\cS$ outside the broader context of its role in Theorem \ref{thm2}.}

\begin{theorem}\label{thm2} 
Let {\red $k=\rdown{cn^{1/3}}$} for some positive constant $c$. Then we have the following.
\begin{enumerate}[(a)]
\item Let $b = b(n) \rightarrow \infty$. Then w.h.p. $r_0 \le r_1 \le b$.
\item \label{partb} W.h.p. $\cI^* \subseteq \cI_k \subseteq \cI^{**}$ where
\begin{align*}
\cI^*&=\cH_{r_0-1} \cup \{E: E \mbox{ contains some  $S \in \cS$ and intersects every edge of $\cH_{r_0-1}$}\},\\
\cI^{**}&=\cH_{r_0-1}\cup \{E: E \mbox{ intersects each $S\in\cS$ and  intersects every edge of $\cH_{r_0-1}$}\}.
\end{align*}
\item W.h.p. $|\cI^{**} \sm \cI^*|=o(|\cI^*|)$ and 
\[
\frac{|\mathcal{I}_k|}{\binom{n}{k}} =(1+o(1)) \frac{\sum_{S \in \mathcal{S}} c^{3(r_0 - 1 - |S|)}}{k^{r_0 - 1}} \text{ as }n\to\infty.
\]
\item \label{r0dist} The sequence
\[{\red 
\left\{ \frac{\lim_{n\to\infty}\Pr(r_0 = i + 1)}{\lim_{n\to\infty}\Pr(r_0 > i + 1)} \right\}_{i \geq 0}}
\]
has exponential generating function 
\[e^x(e^{x^2/(2c^3)} - 1).\]
\end{enumerate}
\end{theorem}
Note that the distribution of $r_0$ can be recovered from part \eqref{r0dist}. Regarding part \eqref{partb}, we note that the containment $\cI^* \subseteq \cI_k \subseteq \cI^{**}$ is not in general strict at either end. For example $\cI^*$ may not be maximally intersecting, and $\cI^{**}$ may not be intersecting at all. We suspect that it may be possible to give a better estimate for where the final $\cI_k$ falls in between $\cI^*$ and $\cI^{**}$, but that this information cannot be determined from the first few edges.

We say that a family $\cJ \subseteq \binom{[n]}{k}$ is a {\it $j$-junta} if there is some $J \subseteq [n]$ with $|J|=j>0$ and a family $\cJ^*$ of subsets of $J$ such that 
 \[
 \cJ = \{T \in \binom{[n]}{k}: T \cap J \in \cJ^*\}.
 \]
 We say that {\red $\cJ$ is {\it generated by $(J,\cJ^*)$}}. Juntas are relevant here because, roughly speaking, they can provide a simple ``certificate'' that a family is intersecting. Specifically if $\cH \subseteq \cJ$ for some junta $\cJ$ generated by $(J,\cJ^*)$ with $\cJ^*$ intersecting, then $\cH$ must be intersecting. For example a trivial intersecting family is a 1-junta, sometimes called a {\it dictatorship}. For a second example a Hilton--Milner-type hypergraph is almost (with the exception of one edge) contained in a dictatorship. Using this terminology, Theorem \ref{thm2} has the following corollary. 
 \begin{corollary}\label{cor1}
 Let $k=\rdown{cn^{1/3}}$ where $c>0$, and $b = b(n) \rightarrow \infty$. Then w.h.p. there is some $b$-junta $\cJ$ such that $|\cI_k \sm \cJ|=o(|\cI_k|)$. Furthermore  $\cJ$ is determined by the hypergraph $\cH_b$.
 \end{corollary}
 
Bohman, Frieze, Martin, Ruszink\'o and Smyth \cite{BCFMR1} considered the case where $n^{1/3}\ll k\ll n^{5/12}$. They prove that w.h.p. the structure of $\cI_k$ satisfies the following: there exists a hypergraph $\cH$ with $\t$ edges and a vertex $v$ such that the following holds: (i)  $n\t^3/6k^3$ converges to the exponential distribution with mean 1 and (ii) $\cI_k$ consists of all $E\in \binom{[n]}{k}$ that (a) contain $v$ and (b) meet every edge of $\cH$ that does not contain $v$. {\red So in this regime we also have that w.h.p. $\cI_k$ is almost (with the exception of the $poly(n)$ edges of $\cH$) contained in a dictatorship. It is interesting to note that, combining the results of \cite{BCFMR}, \cite{BCFMR1} and Theorem \ref{thm2}, for the whole regime $k \ll n^{5/12}$ we have that w.h.p. $\cI_k$ is almost contained in a relatively small junta of order $k^2/n^{1/3}$, and furthermore this small junta is a dictatorship unless $k=\Theta(n^{1/3})$.} This is not a shortcoming in our proof: indeed, for $k=\Theta(n^{1/3})$ and sufficently small fixed $\e>0$, Theorem \ref{thm2} implies that there is a probability bounded away from zero that no dictatorship contains more than $1-\e$ proportion of $\cI_k$.

We also make a little progress on the case where $k\gg n^{1/2}$. In particular we give non-trivial upper and lower bounds on $|\cI_k|$. Let $N=\binom{n}{k}$ and $d=\binom{n-k}{k}$. There is a trivial lower bound of $N/(d+1)$ on the minimum size of a maximal intersecting family. We prove that w.h.p. $|\cI_k|$ is significantly larger.
\begin{theorem}\label{th2}
Suppose that $k\gg n^{1/2}\log^{1/2}n$. Then w.h.p. $\Omega\bfrac{N\log N}{d}\leq |\cI_k|\leq O(N^{1-\eta})$ for some constant $\eta>0$.
\end{theorem}

 Our proof of Theorem \ref{th2} will resemble the analysis of some similar processes that have been studied. Suppose we are given a graph $G$. In the {\it random greedy independent set process}, or just {\it independent process}, we choose a random sequence $(v_1, \ldots, v_m)$ of vertices of $G$ where $v_{i+1}$ is randomly chosen from all vertices not already chosen and not adjacent to any chosen vertex. Then the process we are studying is equivalent to the independent process on the Kneser graph $K(n, k)$ which has vertex set $\set{v_S: S \in \binom{[n]}{k}}$, and where $v_S$ is adjacent to $v_{S'}$ whenever $S \cap S' = \es$. Wormald \cite{Wormald} was the first to study the independent process, which he analyzed on random regular graphs. Lauer and Wormald \cite{LW} extended the analysis from random regular graphs to all regular graphs of sufficiently high girth. Bennett and Bohman \cite{BB} analyzed a generalization of this process to $d$-regular hypergraphs with $d \rightarrow \infty$ sufficiently fast assuming a relatively mild upper bound on codegrees. So, one could hope that we could just use some existing analysis of the independent process and apply it to the Kneser graph. Unfortunately, all the results for general deterministic graphs (or at least the ones we could find) assume a similar upper bound on vertex codegrees which the Kneser graph does not satisfy. Indeed, for say $k=cn$ for constant $c$, the graph $K(n, k)$ has pairs of vertices whose codegree is the same order of magnitude as their degrees. Even worse, every vertex in the whole graph has high codegree with a few other vertices. Thus, our main contribution in this regime is to carry out an analysis of the process that resembles previous work but allows for a few pairs of vertices with high codegree. The analysis uses the so-called differential equation method (see \cite{BD} for a gentle introduction).

To put the above results in some context we point out another related but distinct random model of intersecting families. Improving a result of Balogh, Das, Delcourt, Liu and Sharifzadeh \cite{B0},  Balogh, Das, Liu, Sharifzadeh and Tran \cite{B1} and Frankl and Kupavskii \cite{FK} showed that if $n\geq 2k+\Omega((k\log k)^{1/2})$ then almost all intersecting families are trivial. Balogh, Garcia, Li and Wagner \cite{B2}  reduced the lower bound on $n$ to $2k+100\log k$. Dinur and Friedgut \cite{DF} proved that when $k=pn$ for constant $0<p<1/2$, every intersecting family is (up to a small number of members) contained in an $O(1)$-junta. In \cite{DF} they also proved that when $k=o(n)$ every intersecting family is almost contained in a trivial family (a dictatorship). 

 The organization of the paper is as follows. In Section \ref{sec2} we prove Theorem \ref{thm2}. In Sections \ref{constant} and \ref{smallk} we prove Theorem \ref{th2}, where in each of the two sections we consider a certain range of $k$. The proofs for each regime are almost identical, but the error bounds are somewhat different. In Section \ref{summary} we give some concluding remarks. 
\section{Proof of Theorem \ref{thm2}}\label{sec2}
\subsection{The first $o(\sqrt{\log n})$ steps: generating $\mathcal{S}$}

We examine the first $r$ steps in the process for $r = o(\sqrt{\log n})$. The most important part of the random intersecting process in these first few steps is the generation of a (not necessarily uniform) intersecting hypergraph $\mathcal{S}$ that will be the $\mathcal{S}$ in the definition of $\cI^*$ and $\cI^{**}$. {\blue The random intersecting hypergraph $\mathcal{S}$ in the limit as $n \rightarrow \infty$ is described on its own in the appendix. } To describe $\mathcal{S}$ here we introduce some notation.

Denote by $V(\mathcal{H}_r)$ the set of all vertices of degree at least one in $\cH_r$, and by $U(\mathcal{H}_r)$ the set of vertices of degree exactly one. For any set $S$ of vertices of degree at least two in $\mathcal{H}_r$, let $E_r(S)$ denote the set of edges of $\cH_r$ that contain a vertex in $S$ and let $e_r(S)=|E_r(S)|$. Let $\chi_r(S) = e_r(S) - 2|S|$ and let $\chi^*_r = \max\{\chi_r(S) \mid S \subseteq V(\mathcal{H}_r) \setminus U(\mathcal{H}_r)\}$. At each step $r$, we let 
\[
\mathcal{S}_r = \{S \subseteq V(\mathcal{H}_r) \setminus U(\mathcal{H}_r) \mid \chi_r(S) = \chi^*_r\}
\]
Here we will describe three regimes for $\mathcal{S}_r$: a growing regime, a diminishing regime, and a stable regime. As the name suggests, $\mathcal{S}$ will denote the unchanging $\mathcal{S}_r$ within the stable regime.

The growing regime coincides with all steps before vertices of degree three show up. Letting $r_0$ denote the number of edges when the first vertex of degree three appears we have by the following lemma that w.h.p. for $r < r_0$, $\mathcal{S}_r$ is just the collection of all independent sets of vertices of degree two in $\mathcal{H}_r$ and $\chi^*_r = 0$. {\red A hypergraph $H$ is {\em simple} if $E_1,E_2\in E(H)$ implies that $|E_1\cap E_2|\leq 1$.}

\begin{lemma}\label{simplicityLemma}
If $r=o(\sqrt{\log n})$ and $\mathcal{H}_{r-1}$ is simple and has no vertices of degree at least 3 then the probability that $\mathcal{H}_r$ is not simple is $O(1/k^{1-o(1)})$.
\end{lemma}
This lemma will follow as a corollary to another lemma we prove shortly, so we save the proof for later. By part \eqref{r0dist} of Theorem \ref{thm2}, $r_0$ will take {\red some value bounded by any slowly growing function of $n$}, and we use without further remark that with high probability $r_0$ is smaller than any function that tends to infinity. 

By simplicity, before vertices of degree three appear the size of $\mathcal{S}_r$ is given by the following lemma: a hypergraph will be called intersecting if its edge set defines an intersecting family.
\begin{lemma}\label{number}
If $\mathcal{H}$ is a $k$-uniform simple intersecting hypergraph on $r$ edges with maximum degree two then for each $1 \leq m \leq r/2$ the number of independent sets of vertices of degree two of size $m$ is 
\[\frac{r!}{(r - 2m)! m! 2^m}.\]
\end{lemma}
\begin{proof}
Let $I_{r, m}$ denote the number of independent sets of vertices of degree two of size $m$ when $\mathcal{H}$ satisfies the assumptions. Then, $I_{r, m}$ satisfies the recurrence,
\[I_{r, m} = \frac{\binom{r}{2} I_{r - 2, m - 1}}{m}.\]
Indeed we can pick any of the $\binom{r}{2}$ vertices of degree two to start building an independent set. We then delete the selected vertex and the two edges that contain it and then take an independent set of vertices of degree two of size $m - 1$ from the remaining hypergraph. However, this overcounts by a factor of $m$ because of the choice of the first vertex. 

Now clearly $I_{r, 1} = \binom{r}{2} = \frac{r!}{(r - 2)! 1! 2^1}$. And by induction using the recurrence above we have 
\[
I_{r, m} = \frac{r(r - 1) (r - 2)!}{2m(r - 2- 2m + 2)!(m - 1)!2^{m - 1}} = \frac{r!}{(r - 2m)!m!2^m}=\binom{r}{2m}\binom{2m}{m}2^{-m}.
\]
\end{proof}

We next show that once vertices of degree three appear $\chi^*_r$ starts to increase by one at each step, and that $\mathcal{S}_r$ decays, until it reaches an intersecting family at time $r_1$ with $\mathcal{S} := \mathcal{S}_{r_1}$. 

Before we prove how the random intersecting process stabilizes to $\mathcal{S}$, we introduce some definitions:
\begin{definition}
An edge $E$ that extends the intersecting family $\mathcal{H}_r$ to an intersecting family $\mathcal{H}_{r+1}$ is said to be an \emph{almost simple extension} if for every pair of distinct $x,y\in E$ that are already contained in a common edge in $\mathcal{H}_r$,  both $x$ and $y$ have degree at least two in $\mathcal{H}_r$. 
\end{definition}

\begin{definition}
Given a $k$-uniform intersecting family $\mathcal{H}_r$ on $r$ edges, an edge $E$ that extends the intersecting family is \emph{good} if it is an almost simple extension and the set of vertices $S$ of $\mathcal{H}_r$ of degree at least two that belong to $E$ satisfies $\chi_r(S) = \chi_r^*$. Otherwise we say the extension is \emph{bad}.
\end{definition}

\begin{lemma}\label{GoodExtensions}
Suppose that $r=o(\sqrt{\log n})$ and that $\mathcal{H}_r$ has been generated by the random intersecting process via good extensions with $\mathcal{S}_r$ as defined above, then the probability that $\mathcal{H}_{r + 1}$ is generated by selecting a bad extension of $\mathcal{H}_{r}$ is at most $1/k^{1-o(1)}$.
\end{lemma}
\begin{proof}
We first show that the number of extensions that are not almost simple is at most $k^{-(r + 1 - \chi_r^*+o(1))}\binom{n}{k}$.

We have two cases to consider, either $E$ contains two vertices of degree one in the same edge of $\mathcal{H}_r$ or else it contains one vertex of degree one and one vertex of degree two from an edge of $\mathcal{H}_r$. To build such an extension in the former case we first choose an edge ($r$ choices) and then two vertices of degree one from it (at most $k^2$ choices). Next we choose a set $S$ of vertices of degree at least two to belong to the new edge. Since $\mathcal{H}_r$ has been built out of good extensions, the number of vertices of degree at least two is at most $\binom{r}{2}$. So we have at most $2^{r^2}$ choices for $S$. Once the two vertices of degree one in a common edge and $S$ have been selected for $E$ we have covered $1 + e_r(S)$ edges, so we have $r - e_r(S)- 1$ edges still to cover. We cover these remaining edges with vertices of degree one since $S$ already accounts for the vertices of degree larger than 1 that we will use. So we have to pick at least one vertex from the remaining $r - e_r(S) - 1$ edges. Lastly we choose another $k - (r - e_r(S) - 1) - |S| - 2$ vertices from the ground set. So, using the fact that $k\sim cn^{1/3}$, we see that the number of extensions that are not almost simple is at most
\mults{
rk^22^{r^2}k^{r - e_r(S) - 1} \binom{n}{k - (r - e_r(S) + |S| + 1)} \leq  k^{r - e_r(S) + 1+o(1)} \left( \frac{k}{n}\right)^{r - e_r(S) + |S| + 1}\binom{n}{k} \\
= \frac{1}{k^{r + 1 - (e_r(S) - 2|S|)+o(1)}} \binom{n}{k}
}
This is therefore at most what we claimed since $e_r(S) - 2|S| \leq \chi^*_r$. 

For the other type of not almost simple extensions we have the upper bound of
\[
rkr^2 2^{r^2} k^{r - e_r(S)} \binom{n}{k - (r - e_r(S) + |S| + 1)} \leq \frac{1}{k^{r + 1 - (e_r(S) - 2|S|)+o(1)}} \binom{n}{k}
\]
So again using $e_r(S) - 2|S| \leq \chi^*_r$ we arrive at the same conclusion as in the first case. 

Next we turn our attention to almost simple extensions for which the chosen set $S$ has $\chi_r(S) \leq \chi^*_r - 1$. Let $\nu(S)$ for $S$ a collection of vertices of degree at least two in $\mathcal{H}_r$ be the number of almost simple extensions which contain $S$ and otherwise meet each edge of $\mathcal{H}_r$ at a vertex of degree one. {\red To build such an extension we choose a vertex of degree one from the $r-e_r(S)$ edges not covered by $S$. For each such edge $E$ there are between $k-r^2$ and $k$ choices. This leaves $k-(r-e_r(S))-|S|$ vertices of $E$ to be chosen from $[n]\setminus V(\cH_r)$. }Thus,
\beq{1}{
\brac{k-r^2}^{r-e_r(S)}\binom{n-kr}{k-(r-e_r(S))-|S|}\leq \nu(S)\leq k^{r-e_r(S)}\binom{n}{k-(r-e_r(S))-|S|}.
}
Thus,  since $r=o(k^{1/2})$,
\mult{nu}{
\nu(S)=e^{O(r^2/k)} \frac{k^{2r-2e_r(S)+|S|}}{n^{r-e_r(S)+|S|}}\binom{n}{k}=e^{O(r^2/k)} \frac{k^{2r-2e_r(S)+|S|}}{(c^{-1}k)^{3(r-e_r(S)+|S|)}}\binom{n}{k}\\
=e^{O(r^2/k)} \frac{c^{3(r - e_r(S) + |S|)}}{k^{r - e_r(S) + 2|S|}}\binom{n}{k}\sim \frac{c^{3(r - e_r(S) + |S|)}}{k^{r - e_r(S) + 2|S|}}\binom{n}{k}.
}
Thus the number of almost simple extensions with $\chi_r(S) < \chi^*_r - 1$ is at most 
\[
k^{-(r - e_r(S) + 2|S|+o(1))}\binom{n}{k}\leq k^{-(r + 1+\chi_r^*+o(1))}\binom{n}{k}
\]
The total number of extensions of $\mathcal{H}_r$ is at least the number of good extensions and by our estimate on $\nu(S)$ when $\chi_r(S) = \chi_r^*$ we have the number of good extensions is at least $k^{-(r+\chi_r^* +o(1)))}\binom{n}{k}$ and the claim follows.
\end{proof}
We see that Lemma \ref{simplicityLemma} follows from Lemma \ref{GoodExtensions}.
\begin{proof}[Proof of Lemma \ref{simplicityLemma}]
If $\mathcal{H}_r$ is simple with no vertices of degree at least 3 then $\mathcal{H}_r$ has been generated by good extensions with $\mathcal{S}_r$ as the collection of independent sets of vertices of degree two.  If $\mathcal{H}_r$ has no vertices of degree at least 3 and is simple then the only good extensions of $\mathcal{H}_r$ are simple extensions, thus the probability of a non-simple extension is at most $1/k^{1 - o(1)}$.
\end{proof}

{\red Lemma \ref{GoodExtensions}  implies that w.h.p., step $r_0$ coincides exactly with the first time we take $S \in \mathcal{S}_{r}$ with $S \neq \emptyset$ and take a good extension with $S$ as the set of vertices in $\mathcal{H}_r$ of degree at least two. It follows from \eqref{nu} that $S \in \cS_r$ is selected at each step for the next good extension with probability proportional $(c^{-3})^{|S|}$. So, if $\cH_r$ has maximum degree two and $|\cS_r|=\s=\binom{r}{2}$ then 
\beq{then}{
\Pr(r_0=r+1\mid r_0>r)=\frac{e^{O(r^2/k)}\sum_{s\geq 1}\binom{\s}{s}c^{-3s}}{\sum_{s\geq0 }\binom{\s}{s}c^{-3s}}=e^{O(r^2/k)}\brac{1-\frac{1}{(1+c^{-3})^\s}}.
}
and so for $r=o(\sqrt{\log n})$,
\[
\Pr(r_0>r)\leq\prod_{\r=1}^r\brac{\frac{1}{(1+c^{-3})^{\binom{\r}{2}}}+O\bfrac{r^2}{k}}<\frac{1+O(r^3/k)}{(1+c^{-3})^{r(r+1)/2}}.
\]
Clearly a maximal intesecting family has size at least  $k$. Thus w.h.p. $r_0\leq \om=\om(n)$ for any function $\om\to \infty,\om=o(k)$.

Note that in general $\chi_{r+1}(T)\leq \chi_r(T)+1$ for $T\subseteq \cS_r$. Now for any $T \in \mathcal{S}_{r}$ so that $T \cap S = \emptyset$, with $S$ selected for the good extension to $\mathcal{H}_{r + 1}$ we observe that $\chi_{r+1}(T)=\chi_r(T)$ while for any $T \in \mathcal{S}_{r}$ such that $T\cap S\neq\emptyset$, $\chi_{r+1}(T)=\chi_r(T)+1$. Thus if $T \cap S = \emptyset$ then $ T\notin \mathcal{S}_{r + 1}$ while if $T \cap S \neq \emptyset$ then $T$ remains in $\mathcal{S}_{r + 1}$. 

Suppose now that $\cS_r$ is not intersecting and that $T\in\cS_r$. Then if $S\cap T=\emptyset$ then $T\notin \cS_{r+1}$. So,
\[
\Pr(T\in\cS_{r+1})\leq 1-\frac{e^{O(r^2/k)}c^{-3|T|}}{\sum_{S\in \cS_r} c^{-3|S|}}
\]
and so, assuming $r_1\geq r_0+\r$,
\[
\Pr(T\in\cS_{r_0+\r})\leq \prod_{i=1}^\r \brac{1-\frac{e^{O((r_0+i)^2/k)}c^{-3|T|}}{\sum_{S\in \cS_{r_0}} c^{-3|S|}}}^i\leq \exp\set{-\r\frac{e^{O((r_0+\r)^2/k)}c^{-3\binom{r_0}{2}}}{\sum_{S\in \cS_{r_0}} c^{-3|S|}}}
\]
This implies that if $\r\gg \tfrac{c^{-3\binom{r_0}{2}}}{\sum_{S\in \cS_{r_0}} c^{-3|S|}}$ then $r_1\leq r_0+\r$ w.h.p. So, without further remark we have that with high probability $r_1$ is smaller than any function of $n$ tending to infinity. This verifies part (a) of Theorem \ref{thm2}. 
}
\subsection{Step $o(\sqrt{\log n})$ to step $n^{\omega(1)}$}
We've shown that with high probability in the first $o(\sqrt{\log n})$ steps we have only made good extensions. Moreover we have also described the growing regime, diminishing regime, and stable regime for $\mathcal{S}_r$. We now want to show that conditioned on knowing the final, stable family $\mathcal{S}$, and the two hitting times $r_0$ and $r_1$ (each of which is a random positive integer that is sampled according to some asymptotic distribution) that we have $\cI^*\subseteq \cI_k \subseteq \cI^{**}$, with $\cI^*$ and $\cI^{**}$ depending on $r_0$ and $\cS$. 

In order to prove $\cI^* \subseteq \cI_k \subseteq \cI^{**}$ we first show that for any $r \geq r_1$, the probability that we add $E \in \binom{[n]}{k} \setminus \cI^{**}$ is at most $O\bfrac{1}{n^{M}}$ where $M$ can be set to be any constant.  This means that $\mathcal{H}_r \subseteq \cI^{**}$ for $n^{\omega(1)}$ steps; we handle the rest of the process in the next subsection.

We say that an edge is {\em open} after $r$ steps if it is in $\cH_r$ or if it meets $E_1,E_2,\ldots,E_r$, {\blue otherwise we say that it is closed}, and let $\cO_r$ be the set of open edges at step $r$. By what we showed in the first $o(\sqrt{\log n})$ steps we have $\cH_{r_1} \subseteq \cI^{**}$ w.h.p.  For $r \ge r_1$, all of $\cI^*$ remains open as long as we have $\cH_r \subseteq \cI^{**}$. If an edge $E$ does not meet all of $\cH_{r_1}$ then $E$ is already closed at step $r_1$. Thus the only edges that concern us here are those edges $E$ that meet all of $\cH_{r_1}$ but are not in $\cI^{**}$ because they are disjoint from some $S \in \cS$. 

 We first lower bound $|\cI^*|$. Let $S \in \cS$. Then $S$ meets $E_{r_0}, \ldots E_{r_1}$ and also $2|S|$ of the edges $E_1, \ldots E_{r_0-1}$. Thus to choose an edge $E\in \cI^*$, $E\supseteq S$ we can choose one vertex from each of the $r_0-1-2|S|$ edges of $\cH_{r_1}$ that are disjoint from $S$, and then there is no restriction on the other $k - r_0 + 1 + |S|$ vertices of $E$. So, assuming that $\cH_r\subseteq \cI^{**}$, the number of edges in $\cI^*$ at step $r$ is at least
\beq{low}{
(k-r_0)^{r_0-1-2|S|} \binom{n-r_0k}{k - r_0 + 1 + |S|} \sim k^{r_0-1-2|S|} \brac{\frac{k}{n}}^{r_0-1-|S|} \binom{n}{k} = k^{-(r_0 - 1+o(1))}\binom{n}{k}.
 }
Of course some (at most $r$) of the edges in $\cI^*$ are only open {\red because they are already in $\cH_r$}, but if $r \le n^{\om(1)}$ then this has a negligible effect on the above estimate.

Now we upper bound the number of open edges $E\notin \cI^{**}$ at some step $r$ with $r_1 \ll r \ll \sqrt{\log n}$.  To determine the number of choices for $E$ we first choose some subset $T$ of the vertices of $\mathcal{H}_r$ of degree at least two with $T \cap S = \emptyset$ for some $S \in \cS$. $T$ will be the set of vertices of degree at least two in $\mathcal{H}_r$ that belong to $E$. The number of choices for $T$ is at most $2^{r^2} = n^{o(1)}$. Once $T$ is selected there is $S \in \cS$ with $T \cap S = \emptyset$. {\red For $\ell=\ell(r_1,r_2)$ we have that $\chi_r(T) \leq r - r_0 + 1 - \ell$. Now for each $q$ with $r_1 \leq q = o(\sqrt{\log n})$ we have at least a $\frac{1}{(1+c^{3})^{r_1^2}}$ probability to take $S$ to be the set selected from $\cS$ for the good extension. (The lower bound on probability is derived as in \eqref{then}.) If $\frac{r}{(1+c^{3})^{r_1^2}}\to\infty$ then w.h.p. we pick $S$ at least $\ell$ times between steps $r_1$ and $r$, and at these steps $\chi_r(T)$ does not increase.} Thus at step $r$ we have the number of sets $E \notin I^{**}$ with $T$ as the set of vertices of degree at least two in $E$ is at most
\beq{a}{
k^{r - e_r(T)}\binom{n}{k-(|T| + r - e_r(T))} \leq k^{-(r_0 - 1 + \ell + o(1))}\binom{n}{k}.
}
So at step $r$ the number of open edges $E \notin \cI^{**}$ is at most $n^{-M} \binom{n}{k}$ where $M$ can be set as any large constant with an appropriate choice of $\ell$. However as long as $\cH_r$ is contained in $\cI^{**}$ all edges of $\cI^*$ remain open. For $r \leq o(\sqrt{\log n})$, by step $r$ we have not yet selected an edge not in $\cI^{**}$ and from then on the probability that we pick an edge outside of $\cI^{**}$ is at most $n^{-\omega(1)}$. {\red This follows from \eqref{low} and \eqref{a}.} Thus with high probability $\cH_{n^{\omega(1)}} \subseteq I^{**}$.

\subsection{From step $n^{\omega(1)}$ onward}

In this section we handle the remainder of the process. We will show that for $L$ sufficiently large (depending on $r_0$) $\cO_{n^{L}} \subseteq \cI^{**}$, which together with results in the previous subsection implies that $\cI_k \subseteq \cI^{**}$, and that in turn implies that $\cI^* \subseteq \cI_k$. Let $L = r_0 + 3$. For $E \in \binom{[n]}{k}$ and knowing $r_0$, $r_1$, and $\cS$, let 
\[\cD(E) = \{F \in \cI^* \mid F \cap E = \emptyset\}\]

Suppose $E' \in \cO_{n^L} \sm \cI^{**}$, so that $E'$ is disjoint with some set $S \in \cS$. We lower bound $|\cD(E')|$. Quite crudely, there is at least one choice for some set of vertices $S \in \cS$ that is disjoint with $E'$ and such that $|S|\le r_0$, and from here we can pick from each edge $E_1, \cdots, E_{r_0 - 1}$ not intersecting $S$ a single vertex $v_j \in E_j \setminus E'$. Such a choice always exists since $E' \cap S = 0$ and $E' \neq E_1, \cdots, E_{r_0 - 1}$. So we count edges $E \in \cD(E')$ to be at least 
\[
\binom{n-k-2r_0}{k-2r_0} \ge \frac{1}{k^{4r_0 + o(1)}} \binom{n}{k}
\]
If we pick a set from $\cD(E')$ before we pick $E'$ then $E'$ can never be selected. For $r \leq n^{\omega(1)}$ the probability at each step that we pick a set in $\cD(E')$ is at least
\[
k^{-(3r_0 + 1 + o(1))}.
\]
This estimate again uses the estimate of $k^{-r_0 + 1 + o(1)} \binom nk$ for the total number of extensions. 
Thus the probability that $E'$ remains open at step $n^L$ is at most 
\[(1 - k^{-(3r_0 + 1 + o(1))})^{n^L} \leq \exp\set{-k^{3L - 3r_0 + 1 + o(1)}}.\]
Thus, the expected number of open edges $E' \notin \cI^{**}$ at step $n^{L}$ is at most 
\[
\binom nk \exp\set{-k^{3L - 3r_0 + 1 + o(1)}} \leq \exp\set{3k \log k - k^{3L - (3r_0 + 1 + o(1))}} = o(1),
\]
since $L = r_0 + 3$.
This completes the proof of part (b) of Theorem \ref{thm2}.

\subsection{The size of $\mathcal{I}^*$ and $\mathcal{I}^{**}$}
Theorem \ref{thm2} tells us that our final intersecting family $\mathcal{I}_k$ satisfies 
\[\mathcal{I}^* \subseteq \mathcal{I}_k \subseteq \mathcal{I}^{**}.\]

While in general we do not have a complete description of $\mathcal{I}_k$, we do have enough to conclude that following
\begin{lemma}\label{EnumerationTheorem}
Conditional on the process having reached step $r_1$ with stable family $\cS_{r_1}=\cS$, we have w.h.p.
\[
\frac{|\mathcal{I}_k|}{\binom{n}{k}} = (1+o(1)) \frac{\sum_{S \in \mathcal{S}} c^{3(r_0 - 1 - |S|)}}{k^{r_0 - 1}} \qquad \mbox{ as $n \rightarrow \infty$}.
\]
\end{lemma}
\begin{proof}
It suffices to find a lower bound on $\mathcal{I}^*$ and an upper bound on $\mathcal{I}^{**}$. Both approximations are based largely on the fact that in the growing regime every $S \in \mathcal{S}_r$, $r \leq r_0 - 1$ has $\chi_r(S) = 0$ and in the diminishing regime and stable regime every $S \in \mathcal{S}_r$, $r_0 \leq r$ has $\chi_r(S) = r - r_0 + 1$. Thus $\chi_{r_1}(S) = r_1-r_0 + 1$. We now bound $|\mathcal{I}^{**}|$. For any set $T$ of vertices of degree at least two in $\mathcal{H}_{r_1}$ we have $\chi_{r_1}(T) \leq r_1-r_0 + 1$ with $\chi_{r_1}(T) = r_1-r_0 + 1$ if and only if $T \in \mathcal{S}$. For each choice of $T$ the number of hyperedges of $\mathcal{I}^{**}$ so that $T$ is the set of vertices of degree at least two in $\mathcal{H}_{r_1}$ belonging to the hyperedge is at most
\[k^{r_1 - e(T)} \binom{n}{k - (|T| + r_1 - e_{r_1}(T))}.\]
We sum this over the at most $2^{r^2}$ choices for $T$, but the largest order terms come from $T \in \mathcal{S}$, so we have 
\begin{eqnarray*}
|\mathcal{I}^{**}| &\leq& (1 + o(1))\sum_{S \in \mathcal{S}} \frac{c^{3(r_1 - (e_{r_1}(S) - |S|))}}{k^{r_1 - (e_{r_1}(S) - 2|S|))}}\binom{n}{k} \\
%&=& (1 + o(1))\sum_{S \in \mathcal{S}} \frac{c^{3(r_1 - (r_1-r_0 + 1 + |S|))}}{k^{r_1 - (r_1-r_0 + 1))}}\binom{n}{k} \\
&=& (1 + o(1)) \sum_{S \in \mathcal{S}} \frac{c^{3(r_0 - 1 - |S|))}}{k^{r_0 - 1}}\binom{n}{k},
%&=& O\left(\frac{n^{o(1)}}{k^{r_0 - 1}} \binom{n}{k} \right).
\end{eqnarray*}
after using $\chi_{r_1}(S) = r_1-r_0 + 1$.

On the other hand we have that for any fixed $S \in \mathcal{S}$ the number of sets that contain $S$ and meet all other sets in $\mathcal{H}_{r_1}$ in a vertex of degree one is at least 
\[(k - (r_1)^2)^{r_1 - e(S)} \binom{n - kr_1}{k - (|S| + r_1 - e_{r_1}(S))} = (1 - o(1)) \frac{c^{r_0 - 1 - |S|}}{k^{r_0 - 1}} \binom{n}{k}.\]
We can sum this over all choices of $S \in \mathcal{S}$ to arrive at the conclusion.
\end{proof}
This completes the proof of part (c) of Theorem \ref{thm2}.

%By Theorem \ref{EnumerationTheorem} we have that the asymptotic size of $I_k$ only depends on $r_0$ i.e. the number of steps until vertices of degree three appear. 
We finally consider part (d). The distribution of $r_0$ can be recovered from the following generating function. 

\begin{theorem}\label{GeneratingFunction}
Consider the random intersecting process with {\red $k = \rdown{cn^{1/3}}$ where $c>0$ is a constant} and let $X$ be the random variable counting the number of steps until a vertex of degree three appears, then for $n \rightarrow \infty$ the following sequence
\[{\red
\left\{ \frac{\lim_{n\to\infty}\Pr(X = r + 1 \mid X > r)}{\lim_{n\to\infty}\Pr(X > r + 1 \mid X > r)} \right\}_{r \geq 0}}
\]
has exponential generating function 
\[e^x(e^{x^2/(2c^3)} - 1).\]
\end{theorem}
\begin{proof}
By Lemma \ref{simplicityLemma} we have that if $r = o(\sqrt{\log n})$ {\red then w.h.p. either $\mathcal{H}_r$ has no vertices of degree at least three or it is a simple hypergraph}. We first verify
\begin{lemma}
\beq{x1}{
\frac{\Pr(X = r + 1 \mid X > r)}{\Pr(X > r + 1 \mid X > r)}
}
is asymptotically equal to
\beq{x2}{
\frac{\Pr(X = r + 1 \mid X > r ,SIMPLE(r))}{\Pr(X > r + 1 \mid X > r ,SIMPLE(r))}.
}
\end{lemma}
\begin{proof}
The ratio of the expressions in \eqref{x1}, \eqref{x2} is 
\beq{x3}{
\frac{\Pr(X=r+1,X>r,SIMPLE(r))}{\Pr(X=r+1,X>r)}\times \frac{\Pr(X>r+1)}{\Pr(X>r+1,SIMPLE(r))}.
}
Now we bound the probability $p_r$ that $X > r$ holds but that $SIMPLE(r)$ fails to hold. Clearly 
\[
p_1=\Pr(X > 1 , \neg SIMPLE(1)) \leq \Pr(\neg SIMPLE(1)) = 0.
\]
For $1<r =o(\sqrt{\log n})$ we have
\begin{align*}
p_r&=\Pr(X > r , \neg SIMPLE(r))\\
&= \Pr(X > r , \neg SIMPLE(r) , SIMPLE(r - 1)) + \Pr(X > r , \neg SIMPLE(r - 1)) \\
&\leq \Pr(\neg SIMPLE(r) \mid X > r-1 , SIMPLE(r - 1)) + \Pr(X > r - 1 , \neg SIMPLE(r - 1))\\
&\leq \frac{1}{k^{1-o(1)}}+p_{r-1},
\end{align*}
after using Lemma \ref{simplicityLemma} to bound the first summand. Thus $p_r=O(\sqrt{\log n}/k^{1-o(1)})=O(1/k^{1-o(1)})$. 

Going back to \eqref{x3},
\[
1-\frac{\Pr(X>r,\neg SIMPLE(r))}{\Pr(X=r+1,X>r)}\leq \frac{\Pr(X=r+1,X>r,SIMPLE(r))}{\Pr(X=r+1,X>r)}\leq 1
\]
It follows then that
\[\Pr(X = r +1) \geq  \Pr(X = r +1, SIMPLE(r)).\]
We show later in the proof of Theorem \ref{GeneratingFunction} that 
\[\Pr(X = r + 1, SIMPLE(r)) = \zeta_r\]
for some $\zeta_r > 0$.  So with our estimate on $p_r$ we have that 
\[
\frac{\Pr(X=r+1,X>r,SIMPLE(r))}{\Pr(X=r+1,X>r)}=1-O(k^{-(1-o(1))}).
\]
A similar argument deals with the second product in \eqref{x3}.
\end{proof}
If $\mathcal{H}_r$ is a simple $k$-uniform intersecting hypergraph with maximum degree two then {\red each pair of edges of $\mathcal{H}_r$ intersect in a unique vertex of degree two}. In this case $\cH_{r + 1}$ will have vertices of degree three if and only if we select a nonempty set $S$ of the vertices of degree two of $\mathcal{H}_r$ to belong to the new edge. The event that $S$ is not an independent set of vertices of degree two is negligible as is the event that the extension of $\mathcal{H}_r$ to $\mathcal{H}_{r + 1}$ is not almost simple by Lemma \ref{GoodExtensions}. Thus the number of almost simple extensions that do not add a vertex of degree three is 
\beq{non}{
(1+o(1)) k^r \binom{n}{k - r} \sim k^r \left(\frac{c^3}{k^2}\right)^{r} \binom{n}{k}.
}
On the other hand to count the number of almost simple extensions that do add a vertex of degree three we have a choice of how many vertices of degree three we add. To add $m$ vertices of degree three we must (with high probability) select an independent set of $m$ vertices of degree two in $\mathcal{H}_r$. 

Applying Lemma \ref{number} we see that the number of almost simple extensions of $\mathcal{H}_r$ that create vertices of degree three is 
\beq{do}{
(1+o(1)) \sum_{m = 1}^{\lfloor r/2 \rfloor} \frac{r!}{(r - 2m)! m! 2^m} k^{r - 2m} \binom{n}{k - (r - 2m) - m} \sim \sum_{m = 1}^{\lfloor r/2 \rfloor} \frac{r!}{c^{3m}(r - 2m)!m!2^m} k^r \left(\frac{c^3}{k^2}\right)^r \binom{n}{k}.
}

Thus from \eqref{non} and \eqref{do}, we have for each $r= o(\sqrt{\log n})$, 
\[
\frac{\Pr(X = r + 1 \mid X > r)}{\Pr(X > r + 1 \mid X > r)} \sim  \sum_{m = 1}^{\rdown{r/2}} \frac{r(r - 1) \cdots (r - 2m + 1)}{m! (2c^3)^m}\sim \sum_{m = 1}^{\infty} \frac{r(r - 1) \cdots (r - 2m + 1)}{m! (2c^3)^m}
\]
as $r\to\infty$.

Now we just verify the exponential generating function is what we claimed:
\begin{eqnarray*}
\sum_{r = 0}^{\infty} \sum_{m = 1}^{\infty} \frac{r(r - 1) \cdots (r - 2m + 1)}{r!m!(2c^3)^m} x^r &=& \sum_{m = 1}^{\infty} \sum_{r = 2m}^{\infty} \frac{x^r}{(r - 2m)! m!(2c^3)^m} \\
&=& \sum_{m = 1}^{\infty} \frac{(x^2/(2c^3))^m}{m!} \sum_{r = 2m}^{\infty} \frac{x^{r - 2m}}{(r - 2m)!} \\
&=& \sum_{m = 1}^{\infty} \frac{(x^2/(2c^3))^m}{m!} \sum_{s = 0}^{\infty} \frac{x^s}{s!} \\
&=& \left(\exp\set{x^2/(2c^3)} - 1 \right) e^{x}.
\end{eqnarray*}
\end{proof}
\begin{example}
Using this generating function and Theorem \ref{thm2} we can recover two previous results. Using the generating function we have that $\Pr(r_0 = 3) = \frac{1}{1 + c^3}$. In the case that $r_0 = 3$ we see that $\mathcal{S}$ is consists of just a singleton set and it stabilizes at the third step of the process, so $r_1 = 3$. In this case $\cI^* = \cI^{**}$ and is the star centered at that first vertex of degree three. So this recovers the $c_n \rightarrow c$ case of Theorem \ref{main}.

Moreover we can also recover the Hilton--Milner-type statement of Patk\'os from our methods. If $r_0 = 4$ then at the fourth step of the process we have three edges that all contain a single common vertex and there is a fourth edge not containing that vertex but meeting all three edges. In this case $\cI^* = \cI^{**}$ and is a Hilton--Milner type system. We have
\[\Pr(r_0= 4) = \Pr(r_0 = 4 \mid r_0 > 3)\Pr(r_0> 3) = \Pr(r_0 = 4 \mid r_0 > 3)\Pr(r_0 \neq 3).\]
Using the generating function
\[\Pr(r_0 = 4) = \left(\frac{3}{c^3 + 3}\right)\left(\frac{c^3}{1 + c^3}\right),\]
and so we recover Corollary 1.6 of \cite{Pat}.

%In this case then $H_{r_1}$ is just three edges that meet at a single vertex and so $\mathcal{I}^{*} = \mathcal{I}^{**}$ and both are trivial. Thus we recover the $c_n \rightarrow c$ case of Theorem \ref{main}. Furthermore $\Pr(X = 4) = \left(\frac{3}{c^3 + 3}\right)\left(\frac{c^3}{1 + c^3}\right)$ and then applying Theorem \ref{thm2}(b), we recover Corollary 1.6 of \cite{Pat}.
\end{example}

Naively we might expect that $r_0$ determines $\mathcal{S}$ which determines the final intersecting family. This however is only the case when $r_0$ happens to be small. When $r_0 \leq 6$ it isn't too difficult to see that $\mathcal{S}$ must stabilize to be a star, i.e. it will stabilize as all independent sets of vertices of degree two in $\mathcal{H}_{r_0 - 1}$ that contain some fixed vertex. If $r_0 = 7$ however, we see that we can get a $\mathcal{S}$ that is not a star. If $r_0 = 7$ it is possible that $\mathcal{S} = \{\{u, v\}, \{v, w\}, \{u, w\}, \{u, w, v\}\}$ for $\{u, v, w\}$ an independent set of vertices of degree two in $\mathcal{H}_6$.

This completes the proof of Theorem \ref{thm2}.
\section{Proof of Theorem \ref{th2} when $k =\Theta(n)$}\label{constant} 

In this section we assume $k=cn$ where $c=c(n) \ge \z$ for some constant $\z>0$.
Let $G=K(n, k)$ be the Kneser graph. Recall that $G$ has vertex set $\brac{v_S: S \in \binom{[n]}{k}}$, and $v_S$ is adjacent to $v_S'$ when $S \cap S' = \es$. 

Let $g(x):= x \log (x)$ for $x>0$ and let $g(0)=0$. Stirling's formula gives us that for any constants $a\ge b\ge 0$ we have
\begin{equation}
    \binom{an}{bn} = \exp\{ [g(a) -g(b) - g(a-b)]n + O(\log n)\}.
\end{equation}
Thus $G$ has $N$ vertices and is $d$-regular for
\begin{align*}
    N &= \binom{n}{cn} = e^{p_N(c)n+O(\log n)}\qquad\text{ where }p_N(c) =  -g(c) - g(1-c),\\
    d &= \binom{n-cn}{cn} = e^{p_d(c)n+O(\log n)}\qquad\text{ where }p_d(c) = g(1-c) -g(c) - g(1-2c).
\end{align*}
We claim that $p_d(c)<p_N(c)$ for all $0<c<1/2$. Indeed,  $p_d(0)=p_N(0)=0$. Now since $g'(x)=1+\log x$ we have
\[
\frac{d}{dc}(p_N(c) - p_d(c)) = \frac{d}{dc}(-2g(1-c)+g(1-2c)) = 2\log(1-c) - 2\log(1-2c)>0.
\]
Thus, 
\beq{eps1}{
N^{\e_1} < d < N^{1-\e_1}
}
for some $\e_1=\e_1(c)>0$. 

Now we claim that 
\beq{eps2}{
\text{ each vertex has codegree at most $dN^{-\e_2}$ with all but at most $dN^{-\e_2}$ vertices, for some $\e_2=\e_2(c)$.}
}
 Indeed, suppose $|S \cap S'| \le (c-\d)n$ for some $\d>0$. Then $|S \cup S'| \ge (c+\d)n$, and so the number of sets $S''$ that are disjoint with both $S$ and $S'$ (i.e. the codegree of $v_S, v_{S'}$) is at most 
\begin{equation}\label{eqn:codegbound}
    \binom{n-(c+\d)n}{cn} = e^{p_1(c,\d)n+O(\log n)}\qquad\text{ where }p_1(c,\d)=g(1-c-\d) -g(c) - g(1-2c-\d)]n.
\end{equation}
We claim that $p_1(c, \d)<p_d(c)$ for some $\d=\d(c)>0$. Indeed, we have $p_1(c, 0)=p_d(c)$ and
\[
\frac{\partial p_1}{\partial \d}(c, 0) = -\log(1-c)+\log(1-2c) < 0.
\]
 Thus the codegree bound on line \eqref{eqn:codegbound} is at most $dN^{-\e_2}$ for some $\e_2 = \e_2(c)$. 

Now for a fixed set $S$, the number of $S'$ such that $|S \cap S'| > (c-\d)n$ is at most 
\[
\binom {cn}{(c-\d)n} \binom{n}{\d n} = e^{p_2(c,\d)n+O(\log n)}\qquad\text{ where }p_2(c,\d)=g(c) - g(c-\d) -2g(\d) - g(1-\d).
\]
We claim that for sufficiently small $\d$ we have $p_2(c, \d) < p_d(c)$. This follows from the fact that $p_d(c)>0$, $p_2(c, 0)=0$ and $p_2$ is continuous in $\d$. This verifies \eqref{eps2}.

We let $\e=\min\set{\e_1,\e_2}$ for the remainder of Section \ref{constant}. Both of \eqref{eps1} and \eqref{eps2} remain true with $\e_1$ and $\e_2$ replaced by $\e$ respectively.
\subsection{The good event}\label{sec:GE}
Let $G$ be any $d$-regular graph on $N$ vertices where $N^\e < d < N^{1-\e}$.  Assume that for each vertex $v$, there are at most $dN^{-\e}$ vertices $u$ whose codegree with $v$ is at least $dN^{-\e}$. We run the random greedy independent set process on $G$: initially $I(0)=\emptyset$. After $r$ steps, let $I(r) = \{v_1, \ldots, v_r\}$ be the set of vertices in the independent set constructed so far. Let $V(r)$ be the set of available vertices (i.e. the vertices not adjacent to any vertex in $I(r)$). We choose a random vertex $v\in V(r)$ and put $v_{r+1}=v$. 

Let $D_v(r)$ be the number of available neighbors of $v$, for $v \in V(r)$.\\
Let $B(r)$ be the set of vertices $w \in V(r)$ such that some vertex $u \in I(r)$ has codegree with $w$ at least $dN^{-\e}$.\\
Let $C_v(r) := D_v(r) \cap B(r)$.

Let 
\[
t=t(r) = \frac{dr}{N}.
\]
Define the error function
\begin{equation}
  f(t):=  N^{-\e/20} e^{10t}.
\end{equation}
Let the {\it good event} $\cE_r$ be the event that the following inequalities \eqref{eqn:Vest}, \eqref{eqn:Dest} and \eqref{eqn:Cbound} all hold for all $r' \le r$ and where $t'=t(r')$:
\begin{align}
\card{|V(r')| - Ne^{-t'}} &\le Nf(t') \label{eqn:Vest}\\
\card{|D_v(r')| - de^{-t'}} &\le df(t') \mbox{ for all $v \in V(r') \sm B(r')$} \label{eqn:Dest}\\
|C_v(r')| &\le dN^{-\e/10} \mbox{ for all $v \in V(r')$}. \label{eqn:Cbound}
\end{align}

We will show that w.h.p. the good event $\cE_{r_{end}}$ holds where 
\[
r_{end}:= \frac{\e}{1000} \frac{N \log N}{ d}.
\]
In particular this will imply that w.h.p. the process lasts to at least step $r_{end}$, proving the lower bound on $|\cI_k|$ in Theorem \ref{th2} for the case of $k=cn$. Indeed, 
$$e^{10t}\leq N^{\e/1000}\text{ implying that }f(t) \le N^{-\e/25}=o(1)$$ 
for all $r \le r_{end}$, inequalities \eqref{eqn:Vest} and \eqref{eqn:Dest} imply that for all $r \le r_{end}$ we have
\[
|V(r)| = (1+o(1))Ne^{-t}, \qquad |D_v(r)| = (1+o(1))de^{-t} \mbox{ for all $v \in V(r) \sm B(r).$}
\]
In particular $|V(r_{end})|$ is positive and so the process lasts to step $r_{end}$. The upper bound in Theorem \ref{th2} comes from the upper bound on $|V(r_{end})|$ implied by \eqref{eqn:Vest}.

We will sometimes use the bounds (which hold for all $r \le r_{end}$ in the good event $\cE_r$)
\begin{equation}
    |V(r)| \ge{ N^{1-\e/10}}, \qquad |B(r)| \le N^{1-\e/2}.
\end{equation}
The first bound follows from our estimate of $V(r_{end})$. The second bound follows since $|B(r_{end})| \le r_{end} dN^{-\e}$.

\subsection{Dynamic concentration of $V(r)$}
Here we bound the probability that $\cE_{r_{end}}$ fails due to condition \eqref{eqn:Vest}. We define variables $V^+$ and $V^-$ as follows. 
\beq{V+-}{
V^{\pm}=V^{\pm}(r):=
\begin{cases} 
|V(r)|- N ( e^{-t} \pm f(t) ) & \text{if $\cE_{i-1}$ holds},\\
V^{\pm}(r-1) & \text{otherwise}.
\end{cases}
}
Note that if $\cE_{r_{end}}$ fails due to condition \eqref{eqn:Vest} then we either have that $V^+(r_{end})>0$ or $V^-(r_{end})<0$. To show that those events are unlikely we will establish $V^+$ is a supermartingale, i.e. $\E[\D|V^+(r)| \;| I(r)] \le 0$. Similarly we will show that $V^-$ is a submartingale. 

First we show $\E[\D|V^+(r)| \;| I(r)] \le 0$. If $\cE_r$ fails then $\D V^+(r)= 0$ by definition, so we assume $\cE_r$ holds. We have
\begin{align}
     \E[\D|V(r)| \;| I(r), \cE_r] &=  - \frac{1}{|V(r)|}\sum_{v \in V(r)} (1+D_v(r)) \label{eqn:1stepV}\\
     & \le -\frac{1}{|V(r)|} \brac{|V(r)| - |B(r)|} \brac{de^{-t} - df(t)}\nn\\
     & \le -\brac{1- \frac{N^{1-\e/2}}{N^{1-\e/10}}}\brac{de^{-t} - df(t)}\nn\\
     & \le -d \brac{e^{-t} - f(t) } + O\brac{dN^{-\e/3}}.\nn
\end{align}
We turn to bounding the one-step change in $N ( e^{-t} + f(t) )$ (the deterministic part of $V^+$). We use Taylor's thorem:
\begin{theorem}
Let $g:\mathbb{R}\to\mathbb{R}$ be a function twice differentiable on the closed interval $[a,b]$. Then, there exists a number $\tau$ between $a$ and $b$ such that 
\begin{equation}\label{eq:taylor}
g(b) -g(a)=  g'(a)(b-a) + \frac{g''(\tau)}{2} (b-a)^2.
\end{equation}
\end{theorem}
Since $\D t(r) = \frac dN$, we have for some $\tau \in (t(r), t(r+1))$ that
\begin{align*}
    \D N ( e^{-t} \pm f(t) ) &= N(-e^{-t} \pm f'(t)) \cdot \frac{d}{N} + \frac 12 N(e^{-\tau} \pm f''(\tau)) \cdot \frac{d^2}{N^2}\\
    &= d(-e^{-t} \pm f'(t)) + O\brac{\frac{d^2}{N}}.
\end{align*}
 Thus we have 
\begin{align}
     \E[\D V^+(r) \;| I(r), \cE_r] & \le -d \brac{e^{-t} - f(t) } -d(-e^{-t} + f'(t))+O\brac{dN^{-\e/3}}\label{24}\\
     & = d(f(t) - f'(t)) + O\brac{dN^{-\e/3}}\nn\\
     &= -\Omega\brac{dN^{-\e/20}}.\nn
\end{align}
 Thus $V^+$ is a supermartingale. Now we will show that $V^-$ is a submartingale. From \eqref{eqn:1stepV} we have
 \begin{align}
     \E[\D|V(r)| \;| I(r), \cE_r] &=  - \frac{1}{|V(r)|}\sum_{v \in V(r)} (1+D_v(r))\\
     & \ge -1-\frac{1}{|V(r)|} \Bigg[\brac{|V(r)| - |B(r)|} \brac{de^{-t} + df(t)} + |B(r)|d \Bigg]\nn\\
     & \ge -\brac{de^{-t} + df(t)} - \frac{N^{1-\e/2}}{N^{1-\e/10}}d-1\nn\\
     & \ge -\brac{de^{-t} + df(t)} +O\brac{dN^{-\e/3}}. \nn
\end{align}
Thus, similar to \eqref{24}, we have
\begin{align}
     \E[\D V^-(r) \;| I(r), \cE_r] & \ge -d \brac{e^{-t} + f(t) } -d(-e^{-t} - f'(t))+O\brac{dN^{-\e/3}}\\
     &= \Omega\brac{dN^{-\e/20}}.\nn
\end{align}
We use the following concentration inequality.

\begin{theorem}[Freedman]
Suppose $Y_0,Y_1,\dots$ is a supermartingale such that $\Delta Y_j \le C$ for all $j$, and let $W_m =\displaystyle \sum_{k \le m} \Var [ \Delta Y_k| \mathcal{F}_{k}]$.  Then, for all positive reals~$\lambda$,
\begin{equation}\label{eqn:freedman}
   \Pr(\exists m: W_m \le b \text{ and } Y_m - Y_0 \geq \lambda) \leq \displaystyle \exp\left(-\frac{\lambda^2}{2(b+C\lambda) }\right). 
\end{equation}

\end{theorem}

\textcolor{brown}{We apply Freedman's theorem to the supermartingale $V^+$. In particular we want to bound the probability that it is positive at step $r_{end}$. We have at step $0$ that $V^+(0)=-Nf(0)= -N^{1-\e/20}$. If $V^+(r_{end})$ is positive then $V^+(r_{end}) - V^+(0) > N^{1-\e/20}$. Thus we will use $\lambda = N^{1-\e/20}$}. We bound the one-step change to determine $C$. We have $|\D |V(r)|| \le d$ and  $|\D N ( e^{-t} + f(t) )| \le 2d$ so we can use $C=3d$.  Note that we have
\begin{align*}
  \Var [ \Delta V^\pm(r)| \mathcal{F}_{k}, \cE_k] & = \Var [ \Delta |V(r)| \;| \mathcal{F}_{k}, \cE_k] \le \E[ \Delta |V(r)|^2| \mathcal{F}_{k}, \cE_k] \le d^2.
\end{align*}
Thus we have $W_{r_{end}} \le r_{end} \cdot d^2 = O(Nd \log N)$ and we can take $b=O(Nd \log N)$. Freedman's theorem gives us \textcolor{brown}{that the probability that $V^+(r_{end})$ is positive is at most}
\[
\exp\left(-\frac{\lambda^2}{2(b+C\lambda) }\right) = \exp\brac{-\Omega\brac{\frac{N^{2-\e/10}}{Nd \log N + dN^{1-\e/20}}}} =o(1).
\]
\textcolor{brown}{Similarly one can apply Freedman's theorem to the submartingale $-V^-$ using the same values of $\lambda, C, b$ to show that w.h.p. $-V^-(r_{end})$ is not positive. Thus we have w.h.p. that $\cE_{r_{end}}$ does not fail due to condition \eqref{eqn:Vest}.}

\subsection{Dynamic concentration of $D_v(r)$}
Here we bound the probability that $\cE_{r_{end}}$ fails due to condition \eqref{eqn:Dest}. We define variables $D_v^+$ and $D_v^-$ as follows. 
\beq{Dv+-}{
D_v^{\pm}=D_v^{\pm}(r):=
\begin{cases} 
|D_v(r)|- d ( e^{-t} \pm f(t) ) & \text{if $\cE_{i-1}$ holds and $v \notin B(r)$},\\
D_v^{\pm}(r-1) & \text{otherwise}.
\end{cases}
}
If $\cE_{r_{end}}$ fails due to condition \eqref{eqn:Dest} then we either have that $D_v^+(r_{end})>0$ or $D_v^-(r_{end})<0$ for some $v$. Similarly to the last subsection, we will show these events are unlikely using Freedman's theorem. First we verify that $D_v^+$ is a supermartingale. As before we can assume $\cE_r$ holds. We have
\begin{align*}
   \E[\D|D_v(r)| \;| I(r), \cE_r]& = - \sum_{u \in D_v(r)} \Pr(\mbox{$u \notin V(r+1)$ and $v \notin B(r+1)$})\\
   & \le -\sum_{u \in D_v(r)} \frac{D_u(r) - dN^{-\e}}{|V(r)|}\\
& \le -\frac{(de^{-t} - df(t) - dN^{-\e/10})^2}{Ne^{-t}+Nf(t)}  \\
& = -\frac{d^2}{N}e^{-t} \cdot \frac{(1-e^tf(t) - e^t N^{-\e/10})^2}{1+e^tf(t)}\\
& = -\frac{d^2}{N}e^{-t} \cdot \brac{1 - 3e^tf(t)} + O\brac{d^2N^{-1-\e/15}}
\end{align*}

Meanwhile the one-step change in $d ( e^{-t} + f(t) )$ is 
\[
\frac{d^2}{N} ( -e^{-t} + f'(t) ) +O\brac{d^3N^{-2}}
\]
by Taylor's theorem. Thus we have 
\begin{align*}
     \E[\D D_v^+(r) \;| I(r), \cE_r] & \le -\frac{d^2}{N}e^{-t} \cdot \brac{1 - 3e^tf(t) } -\frac{d^2}{N} ( -e^{-t} + f'(t) ) + O\brac{d^2N^{-1-\e/15}}\\
     & = \frac{d^2}{N}(3f(t) - f'(t)) +  O\brac{d^2N^{-1-\e/15}}\\
     & = -\Omega\brac{d^2 N^{-1-\e/20}}
\end{align*}
 Thus $D_v^+$ is a supermartingale, and similarly $D_v^-$ is a submartingale. Now we apply Freedman's theorem. {\brown As in the last section, bounding the probability that $-D_v^-(r_{end})>0$ is entirely similar to bounding the probability that $-D_v^+(r_{end})>0$, so from here we will only show the work for the latter.} We have 
\begin{align*}
    |\D D_v^+(r)| & \le |\D D_v(r)| + |\D d  e^{-t}| + |\D d f(t) |\\
    &\le dN^{-\e} + O(d^2/N)\\
    & \le 2dN^{-\e}.
\end{align*} 
Thus we can use $C=2dN^{-\e}$. {\brown If $D_v^+(r_{end})>0$ then $D_v^+(r_{end})- D_v^+(0)>df(0) = dN^{-\e/20}$ } and so we use $\l = dN^{-\e/20}$. Note that 
\begin{align*}
  \Var [ \Delta D_v^+(r)| \mathcal{F}_{k}, \cE_k] = \Var [ \Delta |D_v(r)| \; | \mathcal{F}_{k}, \cE_k] &\le \E[ (\Delta |D_v(r)|)^2| \mathcal{F}_{k}, \cE_k] \\
  &\le C \E[\; |\Delta D_v^+(r)| \; | \mathcal{F}_{k}, \cE_k] \\
  & \le C \E[\; |\Delta |D_v(r)|\; | \; | \mathcal{F}_{k}, \cE_k] + C |\Delta d  e^{-t}| + C |\Delta d f(t) |\\
  & = O(d^3 N^{-1-\e}).
\end{align*}
Thus we can take $b=O(r_{end} \cdot d^3 N^{-1-\e}) = O(d^2 N^{-\e/2})$. Freedman's theorem gives us a failure probability of at most 
\[
\exp\set{-\frac{\lambda^2}{2(b+C\lambda) }}= \exp\set{-\Omega\brac{\frac{d^2N^{-\e/10}}{d^2 N^{-\e/2} + dN^{- \e} \cdot dN^{-\e/20}}}} =o\bfrac{1}{N}.
\]
This probability is small enough to beat a union bound over $N$ choices of $v$. Thus we have w.h.p. that $\cE_{r_{end}}$ does not fail due to condition \eqref{eqn:Dest}.
\subsection{The upper bound on $C_v$}
We now bound the probability that $\cE_{r_{end}}$ fails due to condition \eqref{eqn:Cbound}. We define variables $C_v^+$ as follows. 
\beq{Cv+}{
C_v^{+}=C_v^{+}(r):=
\begin{cases} 
|C_v(r)|- d^2 N^{-1-\e/2} i & \text{if $\cE_{i-1}$ holds},\\
C_v^{+}(r-1) & \text{otherwise}.
\end{cases}
}
We show that $C_v^{+}$ is a supermartingale. We have
\[
\E[\D|C_v(r)| \;| I(r), \cE_r] \le \frac{d \cdot dN^{-\e}}{|V(r)|} \le \frac{d^{2}N^{-\e}}{N^{1-\e/10}} \le d^2 N^{-1-\e/2}
\]
and so $\E[\D|C_v^+(r)| \;| I(r), \cE_r] \le d^2 N^{-1-\e/2}-d^2 N^{-1-\e/2}=0$. We apply Freedman's theorem. We have $\D C_v^+(r) \le \D C_v(r) \le dN^{-\e}$ and so we use $C=dN^{-\e}$.  Note that we have
\begin{align*}
  \Var [ \Delta C_v^{+}(r)| \mathcal{F}_{k}, \cE_k] = \Var [ \Delta C_v(r)| \mathcal{F}_{k}, \cE_k]&\le \E[ \Delta C_v(r)^2| \mathcal{F}_{k}, \cE_k]\\
  &\le dN^{-\e} \E[ |\Delta C_v(r)|| \mathcal{F}_{k}, \cE_k]\\
  & \le dN^{-\e} \cdot d^2 N^{-1-\e/2}\\
  & \le d^3 N^{-1-3\e / 2}
\end{align*}
Thus we have $W_{r_{end}} \le r_{end} \cdot d^3 N^{-1-3\e / 2} = O(d^2 N^{-\e})$ and we can take $b=O(d^2 N^{-\e})$. Using $\lambda = dN^{-\e/4}$, Freedman's theorem gives us that the probability we ever have $C_v^+(r) > \lambda$ is at most 
\[
\exp\left(-\frac{\lambda^2}{2(b+C\lambda) }\right) = \exp\brac{-\Omega\brac{\frac{d^2N^{-\e/2}}{d^2 N^{-\e} + dN^{-\e}\cdot dN^{-\e/4} }}} =o(1/N),
\]
which is small enough to beat the union bound over the $N$ choices for $v$. Thus w.h.p. we have for each $v$ that $C_v^+(r) \le \lambda$ and so 
\[
C_v(r) \le d^2 N^{-1-\e/2} r_{end} + \lambda < dN^{-\e/10}.
\]
Thus w.h.p. $\cE_{r_{end}}$ does not fail due to condition \eqref{eqn:Cbound}. This completes our proof of Theorem \ref{th2} for the case when $k=\Theta(n)$.

\section{Proof of Theorem \ref{th2} when $k=o(n)$}\label{smallk}
In this section we assume $\z^{-1} n^{1/2} \log^{1/2} n \le k \le \z n$ for a small constant $\z>0$. We use the same notation as in Section \ref{constant}, so $k=cn$.

For say $|x|\le 1/2$, we have $g(1-x)=-x+\frac{x^2}{2}+O(x^3)$ and so we can write
\begin{align*}
    p_N(c) &=  -g(c) + c - \frac12 c^2 + O(c^3),\\
    p_d(c) &= -g(c) + c - \frac 32 c^2 + O(c^3).\\
    p_N(c)-p_d(c) &= c^2+O(c^3).
\end{align*}
Thus, we have that 
\[
d=Ne^{-(c^2+O(c^3))n}. 
\]
With $p_1,p_2$ as in Section \ref{constant}, we now have
\begin{align*}
p_1(c, \d) - p_d(c)& = -c\d+O(c^3).\\
p_2(c, \d)&=\d \log \brac{\frac{c}{\d^2}}+ 2\d - \frac{\d^2}{c} + O(\d^2).
\end{align*}
 Let $\d = 0.9c$, and 
 \begin{align*}
     \g&:= \frac{d}{N} = \exp\{-c^2n + O(c^3n)\}\\
     \a&:= d\exp\set{-c\d n + O(c^3n) } = d \exp\{-0.9c^2 n + O(c^3 n)\}\\
     \b&:= \exp\set{ \brac{ \d \log \brac{\frac{c}{\d^2}}+ 2\d - \frac{\d^2}{c}}n + O(\d^2n)} =  \exp\set{-0.9 (c \log c) n + O(cn) }
 \end{align*}
Arguing as in Section \ref{constant}, we see that for each vertex $v$, there are at most $\a$ vertices $u$ whose codegree with $v$ is at least $\b$. 
 
 We assume that 
 \[
     \g \le \log^{-100} N
 \]
 which holds when $c$ is at least some large constant times $n^{-1/2} \log^{1/2} n$ (i.e. for $\z$ small enough). 
 \subsection{The good event}

Let $t=t(r) = \g r.$ Define the error function
\begin{equation}\label{eqn:fx}
  f(t):=  \g^{0.1} e^{10t}.
\end{equation}
Note that this is different from our error function $f(t)$ from Section \ref{sec:GE}. Let the {\it good event} $\cE_r$ be the event that the following inequalities \eqref{eqn:Vestx}, \eqref{eqn:Destx} and \eqref{eqn:Cboundx} all hold for all $r' \le r$:
\begin{align}
    \card{|V(r')| - Ne^{-t'}} &\le Nf(t') \label{eqn:Vestx}\\
    \card{|D_v(r')| - de^{-t'}} &\le df(t') \mbox{ for all $v \in V(r') \sm B(r')$} \label{eqn:Destx}\\
    |C_v(r')| &\le \g^{-0.1}\a \mbox{ for all $v \in V(r')$}. \label{eqn:Cboundx}
\end{align}

We will show that w.h.p. the good event $\cE_{r_{end}}$ holds where 
\[
r_{end}:= 0.001 \g^{-1} \log \g^{-1}=\Omega\bfrac{N\log N}{d}.
\]
In particular this will imply that w.h.p. the process lasts to at least step $r_{end}$, proving the lower bound in Theorem \ref{th2} for the case $c_n=o(1)$. Indeed, we have  $e^{10t}\leq \g^{-0.01}$ and so  $e^{-t}\geq \g^{0.001}$ and $f(t) \le \g^{0.09}$
for all $r \le r_{end}$. Now lines \eqref{eqn:Vestx} and \eqref{eqn:Destx} imply that for all $r \le r_{end}$ we have
\[
|V(r)| = (1+o(1))Ne^{-t}, \qquad |D_v(r)| = (1+o(1))de^{-t} \mbox{ for all $v \in V(r) \sm B(r).$}
\]
In particular $|V(r_{end})|$ is positive and so the process lasts to step $r_{end}$. The upper bound in Theorem \ref{th2} comes from the bound on $|V(r_{end})|$ implied by \eqref{eqn:Vestx}.

We will sometimes use the bounds (which hold for all $r \le r_{end}$ in the good event $\cE_r$)
\begin{equation}
    |V(r)| \ge \frac 12 N e^{-\g r_{end}} = \frac12 \g^{0.001} N, \qquad |B(r)| \le \a r_{end} = 0.001 \a \g^{-1} \log \g^{-1}.
\end{equation}

\subsection{Dynamic concentration of $V(r)$}
Here we bound the probability that $\cE_{r_{end}}$ fails due to condition \eqref{eqn:Vestx}. We define variables $V^+$ and $V^-$ exactly as in \eqref{V+-} (but now $f(t)$ is as in \eqref{eqn:fx}). We now
establish that $V^+$ is a supermartingale. 

First we show $\E[\D|V^+(r)| \;| I(r)] \le 0$. If $\cE_r$ fails then $\D V^+(r)= 0$ by definition, so we assume $\cE_r$ holds. We have
\begin{align*}
     \E[\D|V(r)| \;| I(r), \cE_r] &=  - \frac{1}{|V(r)|}\sum_{v \in V(r)} (1+D_v(r)) \\
     & \le -\frac{1}{|V(r)|} \brac{|V(r)| - |B(r)|} \brac{de^{-t} - df(t)}\\
     & \le -\brac{1- \frac{0.001 \a \g^{-1} \log \g^{-1} }{\frac12 \g^{0.001} N}} \brac{de^{-t} - df(t)}\\
     & \le -d \brac{e^{-t} - f(t) } + O\brac{  \a  \g^{-0.001} \log \g^{-1}} .
\end{align*}
We turn to bounding the one-step change in $N ( e^{-t} + f(t) )$. Since $\D t(r) = \frac dN$, we have for some $\tau \in (t(r), t(r+1))$ that
\begin{align*}
    \D N ( e^{-t} + f(t) ) &= N(-e^{-t} + f'(t)) \cdot \frac{d}{N} + \frac 12 N(e^{-\tau} + f''(\tau)) \cdot \frac{d^2}{N^2}\\
    &= d(-e^{-t} + f'(t)) + O\brac{d^2 N^{-1}}.
\end{align*}
 Thus we have 
\begin{align}
     \E[\D V^+(r) \;| I(r), \cE_r] & \le -d \brac{e^{-t} - f(t) } -d(-e^{-t} + f'(t)) + O\brac{  \a  \g^{-0.001} \log \g^{-1}}\label{28x}\\
     & = d(f(t) - f'(t)) + O\brac{  \a  \g^{-0.001} \log \g^{-1}}\nn\\
     &= -\Omega\brac{d\g^{0.1}}, \qquad\text{ since $\a/d\leq \g^{0.8}$}. \nn
\end{align}
Thus $V^+$ is a supermartingale. To see that $V^-$ is a submartingale, we replace \eqref{28x} by
\[
 \E[\D V^-(r) \;| I(r), \cE_r]  \ge -d \brac{e^{-t} - f(t) } -d(-e^{-t} - f'(t)) + O\brac{  \a  \g^{-0.001} \log \g^{-1}}=\Omega\brac{d\g^{0.1}}.
\]

{\brown We apply Freedman's theorem to the supermartingales $V^+$ and $-V^-$ to show w.h.p. neither of them becomes positive. The calculations for $-V^-$ is entirely similar to $V^+$ so we only show the latter. If $V^+(r_{end})>0$ then $V^+(r_{end})-V^+(0) > Nf(0) = N \g^{0.1}$. Thus we use $\lambda=N \g^{0.1} $.} We bound the one-step change to determine $C$. We have $|\D |V(r)|| \le d$ and  $|\D N ( e^{-t} + f(t) )| \le 2d$ so we can use $C=3d$.  Note that we have
\begin{align*}
  \Var [ \Delta V^+(r)| \mathcal{F}_{k}, \cE_k] & = \Var [ \Delta |V(r)| \;| \mathcal{F}_{k}, \cE_k] \le \E[ \Delta |V(r)|^2| \mathcal{F}_{k}, \cE_k] \le d^2.
\end{align*}
Thus we have $W_{r_{end}} \le d^2r_{end}  $ and we can take $b=d^2r_{end} = O\brac{d^2 \g^{-1} \log \g^{-1}}$. Freedman's theorem gives us a failure probability of at most 
\[
  \exp\set{-\frac{\lambda^2}{2(b+C\lambda) }} = \exp\set{-\Omega\set{\frac{ N^2 \g ^{0.2}}{d^2 \g^{-1} \log \g^{-1}+ Nd \g^{0.1} }}} =o(1).  
\]
Thus we have w.h.p. that $\cE_{r_{end}}$ does not fail due to condition \eqref{eqn:Vestx}.
\subsection{Dynamic concentration of $D_v(r)$}
Here we bound the probability that $\cE_{r_{end}}$ fails due to condition \eqref{eqn:Destx}. We define variables $D_v^+$ and $D_v^-$ as in \eqref{Dv+-}. We verify that $D_v^+$ is a supermartingale. As before we can assume $\cE_r$ holds. We have
\begin{align*}
   \E[\D|D_v(r)| \;| I(r), \cE_r]& \leq - \sum_{u \in D_v(r)} \Pr(\mbox{$u \notin V(r+1)$ and $v \notin B(r+1)$})\\
   & \le -\sum_{u \in D_v(r)} \frac{D_u(r) - \a}{|V(r)|}\\
& \le -\frac{(de^{-t} - df(t) - \g^{-0.1}\a)^2}{Ne^{-t}+Nf(t)}  \\
& = -\frac{d^2}{N}e^{-t} \cdot \frac{(1-e^tf(t) - e^t \g^{-0.1}\a d^{-1})^2}{1+e^tf(t)}\\
& = -\frac{d^2}{N}e^{-t} \cdot \brac{1 - 3e^tf(t)+ O\brac{e^t \g^{-0.1}\a d^{-1}}}\\
& = -\frac{d^2}{N}e^{-t} \cdot \brac{1 - 3e^tf(t)} + O\brac{ \g^{0.9}\a }
\end{align*}

Meanwhile the one-step change in $d ( e^{-t} + f(t) )$ is 
\[
\frac{d^2}{N} ( -e^{-t} + f'(t) ) +O\brac{d^3N^{-2}}
\]
by Taylor's theorem. Thus we have 
\begin{align*}
     \E[\D D_v^+(r) \;| I(r), \cE_r] & \le -\frac{d^2}{N}e^{-t} \cdot \brac{1 - 3e^tf(t) } -\frac{d^2}{N} ( -e^{-t} + f'(t) ) + O\brac{ \g^{0.9}\a }\\
     & = \frac{d^2}{N}(3f(t) - f'(t)) +  O\brac{ \g^{0.9}\a }\\
     &  = -\Omega\brac{  \g^{1.1}d }
\end{align*}
 Thus $D_v^+$ is a supermartingale, and similarly $D_v^-$ is a submartingale. 

Now we apply Freedman's theorem. We have 
\begin{align}
    |\D D_v^\pm(r)| & \le |\D D_v(r)| + |\D d  e^{-t}| + |\D d f(t) |\\
    &\le \b + O(d \g )\\
    & \le O(d \g ),
\end{align} 
where on the second line we used that $\D e^{-t} = O(\g e^{-t})= O(\g)$ and $\D f(t) = O(\g f'(t)) = O(\g^{1.1} e^{10t}) = O(\g^{1.09})$. Thus we can use $C=O(d \g )$. {\brown If $D_v^+(r_{end})>0$ then $D_v^+(r_{end})- D_v^+(0)>df(0) = d \g^{0.1}$} and so we use $\l =d \g^{0.1} $. Note that 
\begin{align*}
  \Var [ \Delta D_v^+(r)| \mathcal{F}_{k}, \cE_k] = \Var [ \Delta |D_v(r)| \; | \mathcal{F}_{k}, \cE_k] &\le \E[ (\Delta |D_v(r)|)^2| \mathcal{F}_{k}, \cE_k] \\
  &\le O(d\g) \E[\; |\Delta D_v(r)| \; | \mathcal{F}_{k}, \cE_k] \\
  & =O(d^2 \g^{2}).
\end{align*}
Thus we can take $b=O( d^2 \g^{2} r_{end}) = O( d^2 \g \log \g^{-1})$. Freedman's theorem gives us a failure probability of at most 
\begin{equation}
    \exp\left(-\frac{\lambda^2}{2(b+C\lambda) }\right) = \exp\brac{-\Omega\brac{\frac{d^2 \g^{0.2}}{d^2\g \log \g^{-1} + d^2 \g^{1.1} }}} =o(1/N).
\end{equation}
This probability is small enough to beat a union bound over $N$ choices of $v$. Thus we have w.h.p. that $\cE_{r_{end}}$ does not fail due to condition \eqref{eqn:Destx}.
\subsection{Dynamic concentration of $C_v$}
Here we bound the probability that $\cE_{r_{end}}$ fails due to condition \eqref{eqn:Cboundx}. We define variables $C_v^+$ as in \eqref{Cv+}. We show that $C_v^{+}$ is a supermartingale. We have
\begin{align}
   \E[\D|C_v(r)| \;| I(r), \cE_r]& \le \frac{d \a }{|V(r)|} \le 2d \a \g^{-0.001} N^{-1} = 2\a \g^{0.999}
\end{align}
and so $\E[\D|C_v^+(r)| \;| I(r), \cE_r] \le 0$. We apply Freedman's theorem. We have $\D C_v^+(r) \le \D C_v(r) \le \a$ and so we use $C=\a$.  Note that we have
\begin{align*}
  \Var [ \Delta C_v^{+}(r)| \mathcal{F}_{k}, \cE_k] = \Var [ \Delta C_v(r)| \mathcal{F}_{k}, \cE_k]&\le \E[ \Delta C_v(r)^2| \mathcal{F}_{k}, \cE_k]\\
  &\le \a\E[ |\Delta C_v(r)|| \mathcal{F}_{k}, \cE_k]\\
  & \le 2 \a^2 \g^{0.999}
\end{align*}
Thus we have $W_{r_{end}} =O(\a^2 \g^{-.001}  \log \g^{-1})$ and we can take $b=O(\a^2 \g^{-.001}  \log \g^{-1})$.  Freedman's theorem gives us that the probability we ever have $C_v^+(r) > \lambda := \frac12 \g^{-0.1}\a$ is at most 
\begin{equation}
  \exp\left(-\frac{\lambda^2}{2(b+C\lambda) }\right) = \exp\brac{-\Omega\brac{\frac{\g^{-0.2}\a^2}{\a^2 \g^{-.001}  \log \g^{-1} + \g^{-0.1} \a^2}}} =o(1/N).  
\end{equation}
Otherwise we have $C_v^+(r) \le \frac12 \g^{-0.1}\a$ and so 
\begin{align}
  C_v(r) &\le 2\a \g^{0.999} r_{end} + \frac12 \g^{-0.1}\a \\
  & = O(\a \g^{-.001} \log \g^{-1})  + \frac12 \g^{-0.1}\a\\
  & \le \g^{-0.1}\a.
\end{align}
Thus w.h.p. $\cE_{r_{end}}$ does not fail due to condition \eqref{eqn:Cboundx}. This completes the proof of Theorem \ref{th2}.
\section{Summary}\label{summary}
We have proved some new results about the random intersecting family process. There is still much to do. For example, can we say more about the structure of $\cI_k$ in Theorem \ref{thm2}. Second, as far as Theorem \ref{th2}, we say little about the structure of the final family and there is a large gap between the upper and lower bounds on the family size.

\appendix

\section{Appendix: Describing $\cS$ as a random matching process}
In the proof of Theorem \ref{thm2} we described the process of generating $\cS$ and $r_0$ which determine $\cI^*$ and $\cI^{**}$. Here we isolate just this process away from the broader context to set up an open problem that would have to be solved in order to fully describe the distribution of $\cI_k$ when $k = cn^{1/3}$. Recall that at the beginning of the random intersecting process, before vertices of degree three appear, $\cH_r$ is a simple hypergraph in which every pair of edges intersects in a unique vertex. If we represent the hyperedges as vertices then an independent set of vertices of degree two in $\cH_r$ corresponds to a matching in the complete graph $K_r$. And moreover, once $r_0$ is determined $\cS$ corresponds to an intersecting family $\mathcal{M}$ of matchings on $K_{r_0 - 1}$. For $k = c_n n^{1/3}$, and $c_n \rightarrow c$ we can describe how $\cS$ is generated by Procedure \ref{RandomMatchingAlgorithm} below. {\blue For a collection of finite set $\mathcal{N}$, $\text{rand}_c(\mathcal{N})$ refers to sampling a set $S$ from $\mathcal{N}$ with probability proportional to $c^{|S|}$.}

\begin{algorithm}\label{RandomMatchingAlgorithm}
\SetKwInOut{Input}{Input}\SetKwInOut{Output}{Output}\SetAlgorithmName{Procedure}{}

\Input{$c > 0$}
\Output{A random number $n$ and a family $\mathcal{M}$ of pairwise intersecting matchings of the complete graph $K_n$}
\BlankLine
$n \leftarrow 1$\;
$G \leftarrow K_n$\;
$\mathcal{N} \leftarrow \text{ All matchings on $G$}$\;
$\mathcal{M} \leftarrow \emptyset$\;
\BlankLine
\While{$\mathcal{M} = \emptyset$}{
	$M \leftarrow \text{rand}_c(\mathcal{N})$\;
	\If{$M = \emptyset$}{
		$n \leftarrow n + 1$\;
		$G \leftarrow K_n$\;
	}
	\Else{
		Add $M$ to $\mathcal{M}$\;
		$\mathcal{N} \leftarrow \text{ All matchings on $G$ that intersect $M$}$\;
	}
}
\While{$\mathcal{M}$ is not a maximal intersecting family of matchings on $K_n$}{
	$M \leftarrow \text{rand}_c(\mathcal{N})$\;
	\If{$M \notin \mathcal{M}$}
	{
		Add $M$ to $\mathcal{M}$\;
		$\mathcal{N} \leftarrow \text{ All matchings on $G$ that intersect every matching in $\mathcal{M}$}$\;
	}
}
\Return $\mathcal{M}$ on $K_n$\;
\caption{Random matching procedure}\label{RandomMatching}
\end{algorithm}
By the proof of Theorem \ref{thm2} we can see that $\mathcal{S}$ is generated by Procedure \ref{RandomMatchingAlgorithm} with input equal to $c^{-3}$ for $k = c_n n^{1/3}$, $c_n \rightarrow c$. The situation that Procedure \ref{RandomMatchingAlgorithm} outputs the single edge on $K_2$ corresponds to the case that the system is trivial. The case that Procedure \ref{RandomMatchingAlgorithm} outputs a single-edge matching on $K_3$ corresponds to the Hilton--Milner system described by \cite{Pat}. 

For $K_t$, $t \leq 5$ there is only one combinatorial type for a maximal family of intersecting matching $\mathcal{M}$ on $K_t$, namely all matchings that contain some fixed edge (i.e. a star). For $t = 6$, however, there are two: All matchings that contain some fixed edge and all matchings that contain at least two out of three edges of some perfect matching. For $t = r_0 - 1 \leq 5$ we can therefore determine exactly what $\mathcal{S}$ will look like. When $t = 6$ and $c = 1$, a routine calculation shows that conditioned stopping at $t = 6$, there is a 123/128 chance of getting a star and a 5/128 chance of getting all matchings that contain at least two out of three edges of a fixed perfect matching. As $t$ gets larger these families of matching can become arbitrarily complicated, so a full description of the distribution on $\cS$ would be quite difficult to work out even though we do have a generating function in part (d) of Theorem \ref{thm2} that captures the distribution on $t$.

\end{document}